\documentclass[BCOR8mm,DIV14]{scrartcl}
\usepackage{amsmath, amsthm, amscd, amsfonts, amssymb, graphicx, color}
\begin{document}
\setlength{\oddsidemargin}{0.35in}\setlength{\evensidemargin}{0.35in}  
\setlength{\topmargin}{-.5cm}
\newtheorem{theorem}{Theorem}[section]
\newtheorem{lemma}[theorem]{Lemma}
\newtheorem{proposition}[theorem]{Proposition}
\newtheorem{corollary}[theorem]{Corollary}
\newtheorem{definition}[theorem]{Definition}
\newtheorem{example}[theorem]{Example}
\newtheorem{exercise}[theorem]{Exercise}
\newtheorem{conclusion}[theorem]{Conclusion}
\newtheorem{conjecture}[theorem]{Conjecture}
\newtheorem{criterion}[theorem]{Criterion}
\newtheorem{summary}[theorem]{Summary} 
\newtheorem{axiom}[theorem]{Axiom}
\newtheorem{problem}[theorem]{Problem}
\newtheorem{remark}[theorem]{Remark}
\numberwithin{equation}{section}
\setcounter{page}{1}
\title{The Cauchy-Schwarz Inequality in Complex Normed Spaces}   
\author{VOLKER W. TH\"UREY \quad \\   
   Bremen,  Germany    \thanks{ 49 (0)421  591777, volker@thuerey.de \ . \quad     }   }
 \maketitle
 \begin{abstract}
     We introduce a product in all complex normed vector spaces, 
     which generalizes the inner product of complex inner product spaces. 
     Naturally the question occurs whether the Cauchy-Schwarz inequality is fulfilled.
     We provide a positive answer. This also yields a new proof of the
     Cauchy-Schwarz inequality in complex inner product spaces, which does not rely on the linearity of
     the inner product. The proof depends only on the norm in the vector space. 
     Further we present some properties of the generalized product.    
 \end{abstract}
     {\textit{Keywords and phrases}}: complex normed space, complex inner product space,
                                                   Cauchy-Schwarz inequality   \\ 
     {\textit{AMS subject classification}}:   46B99                                            
  	    \section{Introduction}      \ \ \  \ \ \ \   
  We deal with vector spaces $ X $ over the complex field ${\mathbb C}$, provided with a norm  $ \| \cdot \|$.
  	 As a motivation we  begin with the special case of an inner product space  
  	 $\left( X , <  \cdot \, | \, \cdot > \right) $. The inner product $ <  \cdot \, | \, \cdot > $ 
  	 generates a norm by $ \|\vec{x}\| =  \sqrt{ < \vec{x} \, | \, \vec{x} > } $, for all  $ \vec{x} \in X $. 
  	 By the same token it is well known that the inner  product can be expressed by this norm, namely for 
  	 $ \vec{x} , \vec{y} \in X $ we can write   
	 \begin{align}    \label{allererste definition} 
	       <\vec{x} \ | \ \vec{y}> \ = \  
         \frac{1}{4} \cdot \left[ \: \|\vec{x}+\vec{y}\|^{2} -   \|\vec{x}-\vec{y}\|^{2} \: 
         + {\mathbf{ i}} \cdot \left( \: \|\vec{x} + {\mathbf{ i}} \cdot \vec{y}\|^{2} -  
         \|\vec{x} - {\mathbf{ i}} \cdot \vec{y}\|^{2} \: \right) \: \right] \, , 
   \end{align}      
         where the symbol `${\mathbf{i}}$' means the imaginary unit. \\
    We use  an idea in \cite{Singer} to generate a continuous product in all complex normed vector spaces 
    $ (X, \| \cdot \|)$, which is just the inner product in the special case of a complex inner product space.      
   \begin{definition}   \label{die erste definition}         
               Let   $ \vec{x}, \vec{y}$ be two arbitrary elements 
               of $ X $. In the case of  \ $ \vec{x} = \vec{0} $  or  $ \vec{y} = \vec{0}  $  we set  
               \ $ < \vec{x} \: | \:  \vec{y} > \ := \, 0 $,  \text{and} if both  \
               $ \vec{x} , \vec{y} \neq \vec{0} \ ( \text{i.e.} \ \|\vec{x}\| \cdot \|\vec{y}\| > 0 \; )$ 
      \ we define the complex number   \\  \\
               \centerline { $ < \vec{x} \: | \:  \vec{y} >  \ :=  $  } 
        $$        \|\vec{x}\| \cdot \|\vec{y}\| \cdot  \frac{1}{4} \cdot  
                  \left[ \left\| \frac{\vec{x}}{\|\vec{x}\|}   +  \frac{\vec{y}}{\|\vec{y}\|} \right\|^{2}   - 
                  \left\| \frac{\vec{x}}{\|\vec{x}\|} - \frac{\vec{y}}{\|\vec{y}\|} \right\|^{2} +
                  {\mathbf{ i}} \cdot  \left( \left\| \frac{\vec{x}}{\|\vec{x}\|}  +  
                 {\mathbf{ i}} \cdot \frac{ \vec{y}}{\|\vec{y}\|} \right\|^{2} - \left\| \frac{\vec{x}}{\|\vec{x}\|} -
                 {\mathbf{ i}} \cdot \frac{ \vec{y}}{\|\vec{y}\|} \right\|^{2}  \right)  \right] \ .   $$
   \end{definition}                                                            
     It is easy to show that the product fulfills  the conjugate symmetry  ($ <\vec{x}  \: |  \: \vec{y}>$ 
        $ \ = \ \overline{<\vec{y}  \: |  \: \vec{x}>} $), where $\overline{<\vec{y}  \: |  \: \vec{x}>}$
        means the complex conjugate of  $ <\vec{y} \: | \: \vec{x}> $, the positive definiteness 
        ($ < \vec{x}  \: |  \: \vec{x} > \ \geq \ 0 , \
        \text{and} \ < \vec{x}  \: |  \: \vec{x} > \ = 0 \ \text{only for} \ \vec{x} = \vec{0} $), 
        and the homogeneity for real numbers 
        ($ < r \cdot \vec{x} \, | \, \vec{y} > \ =  r \ \cdot < \vec{x} \, | \, \vec{y} > \ = \ < \vec{x} \, | \,
         r \cdot \vec{y} > $),       
        and the homogeneity for pure imaginary numbers ($  < r \cdot {\mathbf{ i}} \cdot \vec{x} \, | \, \vec{y} > \ 
        = \ r \ \cdot {\mathbf{ i}} \,\cdot <  \vec{x} \, | \, \vec{y} > \
        = \ < \vec{x} \, | \, -r \cdot {\mathbf i} \cdot \vec{y} > $), \ for 
        $ \vec{x},  \vec{y} \in X , \ r \in {\mathbb R}$. \ Further, for $ \vec{x} \in X $ it holds \
        $ \|\vec{x}\| = \sqrt{< \vec{x} \: | \: \vec{x} >}$.    
        
        The product from Definition \ref{die erste definition} opens the possibility to define a generalized
        `angle' both in real normed spaces, see \cite{Thuerey2}, and in complex normed spaces, see
        \cite{Thuerey3}. In this paper we turn our focus on the product. 
        We prove the famous Cauchy-Schwarz-Bunjakowsky inequality, or briefly the Cauchy-Schwarz
        inequality. Further we notice some properties of the product.

 Let $ ( X, \| \cdot \| ) $  be an arbitrary complex normed vector space. 
       In Definition \ref{die erste definition} we defined a continuous product \ 
        $ < \cdot \: | \: \cdot > $ \ on $ X $. 
	     This is an  inner product in the case that the norm  $ \| \cdot \| $
	     generates this product by the equation of line \eqref{allererste definition}. 
                                                        
        Generally, for spaces $ X \neq \{\vec{0}\} $, the codomain of the product from 
        Definition \ref{die erste definition} is the entire complex plane ${\mathbb C}$, i.e. we have a surjective
        map \ $ <  \cdot \, | \, \cdot > : X^{2}  \rightarrow {\mathbb C} $.     
        If we restrict the domain of the product $ <  \cdot \, | \, \cdot > $ on unit vectors of $ X $,
        it is easy to see that the codomain changes into the `complex square' \ 
        $ \{ r + {\mathbf{ i}} \cdot s \in {\mathbb C} \ | \ -1 \leq r, s \leq +1 \} $.
        We can improve this statement: 
        Actually the codomain is the complex unit circle  
        $ \{ r + {\mathbf{ i}} \cdot s \in {\mathbb C} \ | r^{2} + s^{2} \leq 1 \} $. This is a consequence of 
        the Cauchy-Schwarz-Bunjakowsky inequality or `${\mathsf{ CSB }}$ inequality' . 
                          
        First we show that for a proof of this inequality we can restrict our research on the 
        two dimensional complex vector space $ {\mathbb C}^{2} $, provided with all possible norms. 
  \section{General Definitions and Properties} 
	            \ \ \  \ \ \ \       
    Let    \ $ (X , \| \cdot \|) $   be an  arbitrary complex vector space provided with a {\textit{ norm}} 
    $\| \cdot \|$,      this  means  that  there is a continuous  map  
         $\| \cdot \|$: \ $ X \longrightarrow$  ${\mathbb R}^{+} \cup \{ 0 \}$  which fulfills the following axioms \
              $\|z \cdot  \vec{x}\| \ =  |z| \cdot  \|\vec{x}\|$ \ (`absolute homogeneity'),   
              \ $\|\vec{x}+\vec{y}\| \leq  \|\vec{x}\|  + \|\vec{y}\| $ (`triangle \ inequality'), 
              and \ $\|\vec{x}\| = 0 $ only for $ \vec{x}  = \vec{0}$ (`positive definiteness'),
              \ for \  $\vec{x}, \vec{y} \in X$  and $ z \in {\mathbb C}$.  
 
        Let  $ < \cdot \:  | \: \cdot > \ : \  X^{2} \longrightarrow   {\mathbb C} $    be a  map  from  
        the product space  $ X \times X $   into the field   $ {\mathbb C} $. \ Such a map is called  a
	      {\textit{ product}}.  

      Assume that the complex vector space \ $X$ \ is provided with a norm   $ \| \cdot \|$, and further there is  
           a  product  $ < \cdot \: | \: \cdot >: X \times X \rightarrow {\mathbb C}  $. 
           We say that the triple    $( X , \| \cdot \| ,  < \cdot \: | \: \cdot > ) $   
           satisfies the {\textit{ Cauchy-Schwarz-Bunjakowsky Inequality}} or `${\mathsf{ CSB}}$ inequality',
           or briefly the {\textit{ Cauchy-Schwarz Inequality}}, 
           if and only if     for all $ \vec{x}, \vec{y} \in X $   there is the inequality \\  \\
           \centerline{ $ |< \vec{x} \: | \: \vec{y} >| \ \leq \ \|\vec{x}\| \cdot \|\vec{y}\| \; $.  } \\  
                                       
   It is well known that a complex normed space  $ ( X , \| \cdot \|) $, where the product of
       Definition \ref{die erste definition}  
       is actually an inner product, fulfills the ${\mathsf{ CSB}}$ inequality.
                                                                                           
  Let $ ( X, \| \cdot \| ) $  be an arbitrary complex normed vector space. 
       $ < \cdot \: | \: \cdot > $ \ on $ X $. In the introduction 
	     we already mentioned that the product of Definition \ref{die erste definition} is an inner product 
	     in the case that the norm  $ \| \cdot \| $
	     generates this product by the equation in line \eqref{allererste definition}. 
	  \begin{proposition}  \label{proposition drei} 
	      For all vectors \ $\vec{x}, \vec{y} \in (X, \| \cdot \|) $ and for real numbers \ $ r $ \ the product \ 
	      $ < \cdot \: | \: \cdot > $ \ of Definition \ref{die erste definition} has the following properties.  \\
	     $ { \mathrm{(a)}} \  < \vec{x} \: | \: \vec{y}>  \ = \ \overline{<\vec{y} \: | \: \vec{x}>} $ 
	                                                                        \hfill (conjugate symmetry), \\
	     $ { \mathrm{(b)}} \  <\vec{x} \: | \: \vec{x}> \ \ \geq \ 0 , \ \text{and}  \  
                       <\vec{x} \: | \: \vec{x}> \  =  0   \ \text{only for} \   \vec{x} = \vec{0}$ 
	                                                                    \hfill   (positive definiteness), \\
	     $ { \mathrm{(c)}} \  < r \cdot \vec{x} \, | \, \vec{y} > \ = \ r \, \cdot \, <  \vec{x} \, | \, \vec{y} > 
	      \ = \ < \vec{x} \, | \, r \cdot \vec{y} > \ $     \hfill  (homogeneity for real numbers), \\                     
       ${ \mathrm{(d)}} \   < r \cdot {\mathbf{ i}} \cdot \vec{x} \, | \, \vec{y} > \ = \ 
   r \cdot {\mathbf{ i}} \, \cdot <  \vec{x} \, | \, \vec{y} > \ = \
    < \vec{x} \, | \, -r \cdot {\mathbf{ i}} \cdot \vec{y} > $
	              \hfill   (homogeneity for pure imaginary  \ \  \\     {   $ { } $    }      \hfill    numbers),  \\
	      ${ \mathrm{(e)}} \  \| \vec{x} \| = \sqrt{ < \vec{x} \: | \: \vec{x}> }  $
	                                             \hfill   (the norm can be expressed by the product).      
     \end{proposition} 
       \begin{proof}    
           We use Definition \ref{die erste definition}, and the proofs for ${ \mathrm{(a)}}$ and 
            ${ \mathrm{(b)}} $ are easy.
           For positive $r \in {\mathbb R} $ the point  ${ \mathrm{(c)}}$ is trivial.   We can prove 
         $  < - \vec{x} | \vec{y} >  = - <  \vec{x} | \vec{y} >  = <  \vec{x} | -\vec{y} > $,
           and   ${ \mathrm{(c)}}$ follows immediately.  The point  ${ \mathrm{(d)}}$ is similar to 
            ${ \mathrm{(c)}}$, and   ${ \mathrm{(e)}}$ is clear. 
       \end{proof}
    \begin{lemma}   \label{lemma eins}
        For a pair $ \vec{x}, \vec{y} \in X $ of unit vectors, i.e. $  \| \vec{x} \| = 1 = \| \vec{y} \| $,  
        it holds that both the real part and the imaginary part of $ <\vec{x} \: | \: \vec{y}> $ \
        are in the interval $ [ -1, 1 ] $.  
    \end{lemma}      
    \begin{proof} 
         The lemma can be proven easily with the triangle inequality.
    \end{proof}   
    \begin{corollary} \label{corollary always angle}
        Lemma \ref{lemma eins} means, that \
        $ \left\{ <\vec{x} \: | \: \vec{y}>  \ | \ \vec{x},\vec{y} \in X, \ \| \vec{x} \| = 1 = \| \vec{y} \|
        \right\}$  is a subset of the `complex square' \   
        $ \{ r + {\mathbf{ i}} \cdot s \in {\mathbb C} \ | \ -1 \leq r, s \leq +1 \} $. \
        Immediately, it follows for unit vectors $ \vec{x}, \vec{y} $ the estimate \ 
        $ \left| <\vec{x} \: | \: \vec{y}> \right| \ \leq \ \sqrt{2}$. 
    \end{corollary} 
   Now we notice a few facts about the general product $ < \cdot \, | \, \cdot > $ 
   from Definition \ref{die erste definition}.                                
   \begin{lemma} 
      In a complex normed space  $ (X, \| \cdot \|) $ for $ \vec{x}, \vec{y} \in X$ and real $ \varphi $ 
      there are identities   
   $$    < e^{{\mathbf{ i}} \cdot \varphi} \cdot \vec{x} \, | \, \vec{x} > \ = \
           e^{{\mathbf{ i}} \, \cdot \varphi} \cdot   < \vec{ x} \, | \, \vec{x} > \ , \ \text{and} \                            < e^{{\mathbf i} \cdot \varphi} \cdot \vec{x} \, | \,  e^{{\mathbf i} \cdot \varphi} \cdot \vec{y} > 
                     \ = \ < \vec{ x} \, | \, \vec{y} >   \  .                                                  $$
   \end{lemma} 
   \begin{proof} To prove the first equation take an unit vector $ \vec{x} $, and write \ 
         $ e^{{\mathbf i} \, \cdot \varphi} = \cos(\varphi) + {\mathbf i} \, \cdot \sin(\varphi) $, 
         and use Definition \ref{die erste definition}. The second identity comes directly from
         Defintion \ref{die erste definition}.    
   \end{proof}  
   \begin{corollary} For an unit vector $ \vec{x} \in X $ we have that the set
       $ \{ < e^{{\mathbf{ i}} \cdot \varphi} \cdot \vec{x} \, | \, \vec{x} > \ | \  \varphi \in [ 0, 2 \, \pi ] \} $
       is the complex unit circle, since \ 
       $ < e^{{\mathbf{ i}} \cdot \varphi} \cdot \vec{x} \, | \, \vec{x} > \ = \
       e^{{\mathbf{ i}}\ \cdot \varphi} \cdot  < \vec{x} \, | \, \vec{x} > \ = \ e^{{\mathbf{ i}} \cdot \varphi} $. 
   \end{corollary}
     
        The next example shows that in a complex normed space  $ (X, \| \cdot \|) $ 
        generally we have the inequality \ $ < e^{{\mathbf{ i}} \cdot \varphi} \cdot \vec{x} \, | \, \vec{y} > \ 
         \neq \               e^{{\mathbf{ i}} \, \cdot \varphi} \cdot   < \vec{x} \, | \, \vec{y} > $. 
       This statement seems to be `probable', but we need an example, which we yield in the proof of the 
       following lemma. 
                           
    This inequality means, that the set of products \
 $ \left\{ < e^{{\mathbf{ i}} \cdot \varphi} \cdot \vec{x} \, | \, \vec{y} > \ | \ \varphi \in [ 0, 2 \pi ] \right\}$
        \ commonly does not generate a proper Euclidean circle (with radius $  
        \ |< \vec{x} \, | \, \vec{y} >| $) in $ {\mathbb C} $. 
                   
    If we take $ \varphi \in \{ \pi, \pi / 2, $ $ \pi \cdot 3 / 2 \} $, however, we get with  
        Proposition \ref{proposition drei}  three identities  \\
        \centerline{$ < - \vec{x} \, | \, \vec{y} > \ = \ - < \vec{x} \, | \, \vec{y} > \ , \ 
        < {\mathbf{ i}} \cdot \vec{x} \, | \, \vec{y} > \ = \
          {\mathbf{ i}} \, \cdot  < \vec{x} \, | \, \vec{y} > \ , \  \text{and}  \
   < -{\mathbf{ i}} \cdot \vec{x} \, | \, \vec{y} > \ = \ -{\mathbf{ i}} \ \cdot  < \vec{x} \, | \, \vec{y} > $.}            \begin{lemma} 
    In a complex normed space  $ (X, \| \cdot \|) $ generally it holds the inequality   \\
      \centerline {  $ < e^{{\mathbf{ i}} \cdot \varphi} \cdot \vec{x} \, | \, \vec{y} > \ \neq \
                     e^{{\mathbf{ i}} \, \cdot \varphi} \cdot   < \vec{x} \, | \, \vec{y} > $, \
                       even their moduli are different.       }   
   \end{lemma} 
   \begin{proof}  
      We use the most simple non-trivial example of a complex normed space, let \ 
      $ ( X, \| \cdot \|) := $ $ \left( {\mathbb C}^{2} , \| \cdot \|_{\infty} \right) $, where for two complex 
      numbers \ $ r + {\mathbf{ i}} \cdot s \, , \  v + {\mathbf{ i}} \cdot w \in  {\mathbb C} $ \ 
      we have its norm $ \| \cdot \|_{\infty} $ by   
 $$  \left\| \left(  \begin{array}{c} r + {\mathbf{ i}} \cdot s  \\ v + {\mathbf{ i}} \cdot w  \end{array}  \right)  
             \right\|_{\infty} \ = \ \max \left\{ \sqrt{r^{2}+s^{2}} , \sqrt{v^{2}+w^{2}} \right\} \ .   $$ 
      The following calculations are easy, but tiring. We define two unit vectors \ $ \vec{x}, \vec{y} $ \ of \ 
      $ \left(  {\mathbb C}^{2} , \| \cdot \|_{\infty} \right) $, 
  $$   \vec{x} \ := \ \frac{1}{4} \cdot 
          \left( \begin{array}{cr} 1 + {\mathbf{ i}} \cdot \sqrt{15} \\ 2 + {\mathbf{ i}} \cdot 2 \end{array} 
           \right)        \ \quad \text{and} \ \ \   
           \vec{y} \ := \ \frac{1}{4} \cdot 
     \left( \begin{array}{c} 2 + {\mathbf{ i}} \\ 3 + {\mathbf{ i}} \cdot \sqrt{7} \end{array}  \right) \ .       $$
       Some calculations yield the complex number 
    $$ < \vec{x} \, | \, \vec{y} > \ = \ 
             \frac{1}{64} \cdot \left( \ 19 + 4 \cdot \sqrt{7} + 2 \cdot \sqrt{15}
             + {\mathbf{ i}} \cdot \left[ 7 - 4 \cdot \sqrt{7} + 4 \cdot \sqrt{15} \right] \ \right) \ 
             \approx \ 0.583 + {\mathbf{ i}} \cdot 0.186  \ . $$           
    We choose \ $ e^{{\mathbf{ i}} \cdot \varphi} := 1/2 \cdot \left( 1 +  {\mathbf{ i}} \cdot \sqrt{3}\right) $ \ 
        from the complex unit circle, and we get approximately \ 
        $  e^{{\mathbf{ i}} \, \cdot \, \varphi} \, \cdot < \vec{x} \, | \, \vec{y} > \ \approx \ 
        0.130 + {\mathbf{ i}} \cdot 0.598 . $  \ \ After that we take the unit vector 
    \begin{align*}       
       &  e^{{\mathbf{ i}} \cdot \varphi} \cdot \vec{x} \, = \, \frac{1}{8} \cdot 
          \left( \begin{array}{c} 1 - \sqrt{45} + {\mathbf{ i}} \cdot \left[ \sqrt{3} + \sqrt{15} \right]   \\
          2 - 2 \cdot \sqrt{3} + {\mathbf{ i}} \cdot \left[ 2 + 2 \cdot  \sqrt{3} \right] \end{array}  \right) , \   
          \text{and we compute the product} \ < e^{{\mathbf{ i}} \cdot \varphi} \cdot \vec{x} \, | \, \vec{y} > ,  \\ 
       &  < e^{{\mathbf{ i}} \cdot \varphi} \cdot \vec{x} \, | \, \vec{y} > \  = \ 
          \left( p \ + {\mathbf{ i}} \cdot q \right)/ 64 \ \approx \ 0.113 + {\mathbf{ i}} \cdot 0.628 , \
          \text{where} \ p \ \text{and} \ q \ \text{abbreviate real numbers}   
    \end{align*}     
   $$      p  =  11 + 2 \cdot \left( \sqrt{7} + \sqrt{21} - \sqrt{45} \right) - 5 \cdot \sqrt{3} + \sqrt{15} ,
       \ \ q   =  8 + 2 \cdot \left( 4 \cdot \sqrt{3} - \sqrt{7} + \sqrt{15} + \sqrt{21} \right) + \sqrt{45} \ . $$
      This  proves the inequality \
        $ < e^{{\mathbf{ i}} \cdot \varphi} \cdot \vec{x} \, | \, \vec{y} > \ \neq \
          e^{{\mathbf{ i}} \, \cdot \varphi} \cdot   < \vec{x} \, | \, \vec{y} > $, \ and the lemma is confirmed.    
   \end{proof} 
    The above lemma suggests the following conjecture. One direction is trivial.
    \begin{conjecture}
          In a complex normed space  $ (X, \| \cdot \|) $  for all \ 
          $ \vec{x}, \vec{y} \in X , \, \varphi \in {\mathbb R} $, it holds \  
      \centerline {  $ < e^{{\mathbf{ i}} \cdot \varphi} \cdot \vec{x} \, | \, \vec{y} > \ = \
                     e^{{\mathbf{ i}} \, \cdot \varphi} \, \cdot \,  < \vec{x} \, | \, \vec{y} > $  }  \\
                     if and only if its product  $ < \cdot \; | \; \cdot >  $ from
          Definition \ref{die erste definition} is actually an inner product, i.e. 
          $ (X,  < \cdot \; | \; \cdot >) $ is an inner product space.  
    \end{conjecture}   
  \section{The Cauchy-Schwarz-Bunjakowsky Inequality}    \label{section three} 
     In this section we deal with the famous Cauchy-Schwarz-Bunjakowsky inequality or 
     `${\mathsf{ CSB }}$ inequality' , 
     or briefly the Cauchy-Schwarz inequality. Another name is the `Polarization Inequality'.
     Let $ X $ be a complex normed space, let  $\| \cdot \| $ be the norm on $ X $ and let $ < \cdot \: | \: \cdot > $
     be the product from Definition \ref{die erste definition}. 
     We ask whether  in the triple   $( X , \| \cdot \| ,  < \cdot \: | \: \cdot> ) $  the inequality 
  \begin{align}    \label {CSB Ungleichung}
     |< \vec{x} \: | \: \vec{y} >| \ \leq \ \|\vec{x}\| \cdot \|\vec{y}\| \; 
   \end{align}       
     is fulfilled for all $ \vec{x}, \vec{y} \in X $. The answer is positive. 
   \begin{theorem}  \label{main theorem}
       The Cauchy-Schwarz-Bunjakowsky inequality in line \eqref{CSB Ungleichung} holds in all 
        complex vector spaces $ X $, provided with a norm $ \| \cdot \| $ 
        and the product $ < \cdot \: | \: \cdot > $ from Definition \ref{die erste definition}.            
   \end{theorem}  
  \begin{remark}   
    This theorem is the main contribution of the paper. 
    The proof of the  Cauchy-Schwarz inequality in inner product spaces is well documented in many books 
    about functional analysis by using the linearity of the inner product,   
    see for instance \cite{Werner}, p.204. This new proof of the Cauchy-Schwarz inequality depends only 
    on the norm in the vector space.   
   \end{remark}   
   \begin{proof} 
    First we need a lemma, which shows that for a complete answer it suffices to investigate the 
    complex vector space ${\mathbb C}^{2}$, provided with all possible norms. 
  \begin{lemma} The following two statements ${\mathrm{(1)}}$ and ${\mathrm{(2)}}$ are equivalent.     \\  
    ${\mathrm{(1)}}$ There exists a complex normed vector space $ ( X , \| \cdot \| ) $ and two vectors 
      $ \vec{a} , \vec{b} \in X $ \ with 
    \begin{equation}  \label{erste ungleichung}
        \ |< \vec{a} \: | \: \vec{b} >| \ > \ \|\vec{a}\| \cdot \|\vec{b}\| \ .  
    \end{equation}    
   ${\mathrm{(2)}}$ There is a norm $\| \cdot \| $ on  ${\mathbb C}^{2}$ and two unit vectors $ \vec{x} , \vec{y} 
                \in {\mathbb C}^{2} $ \ with  
     \begin{equation}  \label{zweite ungleichung}
         |< \vec{x} \: | \: \vec{y} >| \ > \ 1 \ . 
     \end{equation}       
   \end{lemma}      
  \begin{proof}  \qquad (1) \ \ $\Leftarrow$  \ \   (2)   \  \  Trivial.  \\
         \qquad  \qquad               (1) \ \ $\Rightarrow$ \ \   (2)   \  \ Easy. \ 
         Let us consider the two-dimensional subspace $\mathsf{{U}}$ of $ X $ which is spaned by the linear 
         independent vectors $ \vec{a} , \vec{b} $. This space $\mathsf{{U}}$ is isomorphic to 
         ${\mathbb C}^{2}$. We take the norm from $ X $ on $\mathsf{{U}}$.
         We normalize $ \vec{a} , \vec{b} $, i.e. we define unit vectors $ \vec{x} := \vec{a}/\| \vec{a} \| $, and 
         $ \vec{y} := \vec{b}/\| \vec{b} \| $. 
         Hence the inequality \eqref{erste ungleichung} turns into \eqref{zweite ungleichung}. 
  \end{proof} 
     The lemma means, that we can restrict our investigations on the complex vector space  ${\mathbb C}^{2}$.
      By a transformation of coordinates we state that  instead of the unit vectors $ \vec{x} , \vec{y} $ 
     of inequality \eqref{zweite ungleichung} we set $ (1,0) := \vec{x}$, and $ (0,1) := \vec{y} $.
      With Definition \ref{die erste definition} the product \; $ < (1,0) | (0,1) >  $ \; has the presentation
      \begin{equation}       \label{produktgleichung}
	            \left<       \left( \begin{array}{c} 1 \\ 0 \end{array} \right) \ | \ 
	                         \left( \begin{array}{c} 0 \\ 1 \end{array}  \right)   \right> 
	              \ = \ \    \frac{1}{4} \ \cdot \
	              \left[ \;  \left\| \left(  \begin{array}{c} 1 \\ 1  \end{array}  \right) \right\|^{2}   \ - \ 
	                         \left\| \left(  \begin{array}{c} 1 \\ -1 \end{array}  \right) \right\|^{2} 
	                         \ + \ {\mathbf{ i}} \ \cdot 
	       \left( \left\| \left( \begin{array}{c} 1 \\ {\mathbf{ i}} \end{array}  \right)    \right\|^{2}   \ - \
	                         \left\| \ \left( \begin{array}{c} 1 \\ -{\mathbf{ i}}  \end{array} \right) \right\|^{2} 
	                                                                                     \right) \;  \right] \ . 
	    \end{equation}  
	   We take four suitable real numbers \ $ s, t, v, w $, and we define four positive values     
       \begin{equation}       \label{normen}
	          \left\| \left(  \begin{array}{c} 1 \\ 1  \end{array}  \right) \right\|        \ =: \ \frac{1}{s} \ , \ \ 
	          \left\| \left(  \begin{array}{c} 1 \\ -1  \end{array}  \right) \right\|       \ =: \ \frac{1}{t} \ , \ \ 
	   \left\| \left(  \begin{array}{c} 1 \\ {\mathbf{ i}}  \end{array}  \right) \right\|   \ =: \ \frac{1}{v} \ , \ \ 
	          \left\| \left(  \begin{array}{c} 1 \\ -{\mathbf{ i}} \end{array}  \right) \right\| \ =: \ \frac{1}{w} \ , 
     \end{equation}  
     or, equivalently, we have four unit vectors   \ \ \
    $(s , s), (t , -t), (v , {\mathbf{ i}} \cdot v ), (w , -{\mathbf{ i}} \cdot w ), \ \text{ i. e.}$  \\  
   \centerline{ $1 = \|(s , s) \| =  \|(t , -t) \| =  \|(v , {\mathbf{ i}} \cdot v ) \| = 
                     \|(w , -{\mathbf{ i}} \cdot w ) \|$.}   \\ 
    Hence the product \; $ < (1,0) | (0,1) >  $ \; changes into  
     \begin{equation}      \label{produktgleichungzwei}
	            \left<       \left( \begin{array}{c} 1 \\ 0 \end{array} \right) \ | \ 
	                         \left( \begin{array}{c} 0 \\ 1 \end{array}  \right)   \right> 
	              \ = \ \    \frac{1}{4} \ \cdot \
	              \left[ \;  \left( \frac{1}{s} \right)^{2}   \ - \   \left( \frac{1}{t} \right)^{2} 
	                         \ + \ {\mathbf{ i}} \cdot  \left( 
	                         \left( \frac{1}{v} \right)^{2}   \ - \  \left( \frac{1}{w} \right)^{2}  
	                                                                         \right) \;  \right] \ .      
	     \end{equation}       
    Further, instead of the  ${\mathsf{ CSB }}$ inequality $ |< (1,0) | (0,1) >| \leq 1 $, for an easier handling
    we can deal with the equivalent inequality  
     \begin{equation}   \label{abschaetzungCSB}
              \left| 4 \cdot \left< \left( \begin{array}{c} 1 \\ 0 \end{array} \right) \ | \ 
	                           \left( \begin{array}{c} 0 \\ 1 \end{array}  \right)   \right> \right|^{2} 
	      \ = \     \left[ \;  \left( \frac{1}{s} \right)^{2}   \ - \   \left( \frac{1}{t} \right)^{2} \;  \right]^{2}
	                         \ + \     
	                 \left[    \left( \frac{1}{v} \right)^{2}   \ - \  \left( \frac{1}{w} \right)^{2}  
	                                  \;  \right]^{2} \ \ \leq \ \ 16 \ .  
	   \end{equation}
   \begin{lemma}  \label{allererstes Lemma}
	                All four numbers \ $ s, t, v, w $ \ are greater or equal $ 1/2 $.
	 \end{lemma}
	   \begin{proof}   
	              For instance to show $ 1/2 \leq s $, use the equation $ (s,s) = s \cdot (1,0) + s \cdot (0,1) $. 
	              Apply the triangle inequality, and note \ 
	              $ \| (s,s) \| = 1 $, \ and also \ $ \| (0,1) \| = 1 = \| (1,0) \| $. 
	   \end{proof}             
	      The next lemma means, that we can assume that both the real part and the imaginary part of
	      \, $ < (1,0) | (0,1) > $ \, are positive.   
	   \begin{lemma}    \label{zweites Lemma}
	              Without restriction of generality we assume $ s < t $ and $  v < w $.
     \end{lemma}
	   \begin{proof}
	       In the case of $ s = t $, the  first sumand of the middle term in line \eqref{abschaetzungCSB} is zero.  
	       From Lemma \ref{allererstes Lemma} follows $ 1/v \leq 2 $. Hence 
	       $ | 4 \cdot < (1,0) | (0,1) > |^{2} \leq  (1/v)^{4} \leq (2)^{4} = 16 $, it holds                                     \eqref{abschaetzungCSB}.  
	                	         
	      In the case of $ s > t , \ \text{i.e.} \ 1/s < 1/t $, i.e. we have a negative real part of 
	        $ < (1,0) | (0,1) > $, we consider instead $ < (-1,0) | (0,1) > $. 
	        By Proposition \ref{proposition drei}(c), we get a positive real part. With a transformation of
	        coordinates we rename $ -(1,0) $ into $ (1,0) $, to get a representation $ < (1,0) | (0,1) > $
	        with positive real part. In the case that the imaginary part of $ < (1,0) | (0,1) > $ 
	        is still negative, we take the product $ < (0,1) | (1,0) > $.
	        Now, by Proposition \ref{proposition drei}(a), also the imaginary part is positive. 
	        We make a second transformation of coordinates, and in new coordinates we call this $ < (1,0) | (0,1) > $.  
	   \end{proof}
	     The following propositions Proposition \ref{proposition neun} and Proposition \ref{entscheidene proposition} 
	     collect general properties of the product   
	     $ < (1,0) | (0,1) > $  from line \eqref{produktgleichung}. The proofs always
	     rely on the triangle inequality of a normed space, which is equivalent to the fact
	     that its unit ball is convex.   
 	 
	 The next proposition looks weird, but it will give the deciding hint for the proof. 
	 \begin{proposition}  \label{proposition neun}
	  	        We get for each $ b \in  {\mathbb R} $ the following two inequalities.
	     \begin{eqnarray}   \label{wichtige Ungleichung}
              \frac{1}{s} \ \leq \ 
               2 \cdot \sqrt{1 + 2 \cdot b^{2} - 2 \cdot b}  + \sqrt{2} \cdot |b| \cdot \frac{1}{w} \ , \ \ \
               \frac{1}{v} \ \leq \ 
               2 \cdot \sqrt{1 + 2 \cdot b^{2} - 2 \cdot b} + \sqrt{2} \cdot |b| \cdot \frac{1}{t} \ . 
	     \end{eqnarray}                    
 	 \end{proposition} 
	   \begin{proof} Please see  the following Proposition \ref{proposition drei Punkt drei}.
	    From the line \eqref{Abchaetzung 1durch s} we get \\
	       \centerline{  $1/s \ \leq \ \sqrt{(1-a)^2 + b^2} \ + \  \sqrt{(1-b)^2 + a^2} \ + \ 
	                     \sqrt{a^2 + b^2} \, / w $ ,    }  \\
	      which is true for arbitrary real numbers $ a , b $. If we choose $ b = a $, it follows the first inequality
	      of Proposition \ref{proposition neun}. 
	      The second inequality uses the corresponding equation of line \eqref{noch eine Gleichung} .     
	  \end{proof} 
    The above Proposition \ref{proposition neun} has an important consequence. 
    \begin{proposition}   \label{entscheidene proposition}  
       If \ $ \frac{1}{2} \cdot \sqrt{2} \leq w $ \ it holds inequality ${\mathsf{ (A)}}$, and in the case of
          \ $ \frac{1}{2} \cdot \sqrt{2} \leq t $ \ it holds inequality ${\mathsf{ (B)}}$, where
	 $$ {\mathsf{ (A)}} : \  \left(\frac{1}{s}\right)^{2} \ \leq \ 2 + \frac{\sqrt{4 \cdot w^{2} - 1}}{ w^{2}} 
	     \qquad  \text{and} \qquad \
	    {\mathsf{ (B)}} : \ \left(\frac{1}{v}\right)^{2} \ \leq \ 2 + \frac{\sqrt{4 \cdot t^{2} - 1}}{ t^{2}}  \ .  $$  
	  \end{proposition} 
	  \begin{proof}
	       To prove inequality ${\mathsf{ (A)}}$ \, we consider again Proposition \ref{proposition neun}, 
	       and we investigate the right hand side of the first inequality in line \eqref{wichtige Ungleichung}. 
	       For all constants $ w \geq 1/2 $, we define a function ${\cal R}(b), \ b \in  {\mathbb R}$, 
	    \begin{align}  \label{definition R von b}    
	       {\cal R}(b) 
	        \ := \  2 \cdot \sqrt{1 + 2 \cdot b^{2} - 2 \cdot b} \ + \ \sqrt{2} \cdot |b| \cdot \frac{1}{w} \ . 
	    \end{align}    
	       Obviously we get the limits \ $ \lim_{b \rightarrow +\infty} {\cal R}(b) = \lim_{b \rightarrow
	        -\infty} {\cal R}(b) = +\infty $  and since the parabola $ 1 + 2 \cdot b^{2} - 2 \cdot b $
	         has only positive values, we state that $ {\cal R} $ has a codomain of positive numbers,   \\
	   \centerline{  ${\cal R}: {\mathbb R} \rightarrow   {\mathbb R}^{+} \ . $  }       \\ 
	       By Proposition \ref{proposition neun} it holds $ 1/s \leq {\cal R}(b)$ for all $ b $, hence 
	       we are interested in minimums of $  {\cal R} $, to get an estimate for  $ 1/s $ as small as possible.  
	       Since $ {\cal R}(-b) > {\cal R}(b) $ for all positive $ b $, the minimum must occur at a 
	       non negative $ b $. Therefore, we consider the function $ {\cal R} $ for non negative $ b $. 
	       The search for a minimum is the standard method, we have 
	   $$ {\cal R}'(b) \ = \ \frac{ 4 \cdot b - 2}{\sqrt{1 + 2 \cdot b^{2} - 2 \cdot b} } + \frac{\sqrt{2}}{w} \ ,
	                                                  \ \ \   \text{for all} \ b \geq 0 \ .             $$
	      In the case of \ $ \frac{1}{2} \cdot \sqrt{2} < w $ \  the equation \ $ {\cal R}'(b_E) = 0 $
	      \ has one positive solution \ $ b_{E} $, 
	   $$  b_{E} \ = \ \frac{1}{2} \cdot \left[ 1 - \frac{ 1 }{\sqrt{4 \cdot w^{2} - 1} } \right] \ \qquad  
	         ( \ \text{We have} \ w > 1/2 \ \text{by the lemmas} \ \ref{allererstes Lemma} \
	                    \text{and} \ \ref{zweites Lemma} \ ).                                $$
	      Recall that we are looking for positive $ b_{E} \, 's $, hence we are investigating the positive part 
	      in the definition \eqref{definition R von b} of $ {\cal R}(b) $, i.e. $ b \geq 0 $. 
	      Note that the condition $  0 \leq  b_{E}  $ holds if $ \frac{1}{2} \cdot \sqrt{2} \leq w $. 
	      We add this as an assumption, i.e. in this Proposition  we assume \ $ \frac{1}{2} \cdot \sqrt{2} \leq w, t $. 
	                                                               
	           As an intermediate step we mention that for the term \ $ 1 + 2 \cdot b^{2} - 2 \cdot b $ \, for \,
	           $ b = b_{E} $ \, we get the value 
	     $$  \frac{2 \cdot  w^{2} }{4 \cdot  w^{2} - 1} \ .    $$    
	      Finally we get an expression of $ {\cal R} \ \text{at} \ b_{E} $, we have 
	      $$   \left( b_{E},  {\cal R}(b_{E}) \right) \ = \ 
	      \left( \frac{1}{2} \cdot \left[ 1 - \frac{ 1 }{ \sqrt{ 4 \cdot w^{2} - 1} } \right] \, , \ \ 
	      \frac{\sqrt{2}}{2 \cdot w} \cdot \left[ 1 +  \sqrt{ 4 \cdot w^{2} - 1}  \right] \right) \  .    $$
	      By Proposition \ref{proposition neun} we get an estimate for $ 1/s $, but actually we are more 
	      interested in an estimate for $ (1/s)^{2} $. We calculate  
	      $$  \left( {\cal R}(b_{E}) \right)^{2}  \ = \ 2 + \frac{ \sqrt{ 4 \cdot w^{2} - 1}}  {w^{2}} \ ,  $$
	      which finishes the proof of Proposition \ref{entscheidene proposition}. 
	       \end{proof} 
       The  above propositions may be a useful tool  for further computations, but we do not know
	           whether the list is complete. For our purpose it will be sufficient. With   
	           Proposition \ref{entscheidene proposition} we are able to do the final stroke. 
	     We are still proving Theorem \ref{main theorem}, i.e. we try to confirm the Cauchy-Schwarz
	     inequality \eqref{abschaetzungCSB}. To prove the theorem, we need to distinguish between  three cases
	     $ \mathsf{ (Case \, a), (Case \, b),  (Case \, c)} $; only the third will be difficult.  \\  \\ 
	     $ \mathsf{ (Case \, a) }$: Let both  $ t, w $ be in the closed interval \
	     $ \left[1/2, \, \sqrt{2} / 2 \right] $. 
	     Hence \  $ | 4 \cdot < (1,0) | (0,1) >|^{2} =  $  \     
	     $\left[ (1/s)^2 - (1/t)^{2} \right]^{2} +  \left[(1/v)^2 - (1/w)^{2}\right]^{2} \, \leq  \,
	      2 \cdot \left[ 4 - (2/\sqrt{2})^{2} \right]^{2} \, = \, 2 \cdot \left[ 2 \right]^{2} = 8 \ .   $   \\ \\
	      $ \mathsf{ (Case \, b) }$: Let \ $ 1/2 \, < \, t \, \leq \, \sqrt{2} / 2  \, < \, w $ \ or the contrary \
	      \ $ 1/2 \, < \, w \, \leq \, \sqrt{2} / 2  \, < \, t $.  \\
	    We use the estimate ${\mathsf{ (A)}}$ or ${\mathsf{ (B)}}$  from Proposition \ref{entscheidene proposition},
	    and we compute  
	    \begin{align}  
	       \left| 4 \cdot \left< \left( \begin{array}{c} 1 \\ 0 \end{array} \right) \ | \ 
	                           \left( \begin{array}{c} 0 \\ 1 \end{array}  \right)   \right> \right|^{2} 
	        \leq \ &   \left[ \ 2 + \frac{\sqrt{4 \cdot w^{2} - 1}}{ w^{2}} \ - 
	          \ \left(\frac{2}{\sqrt{2}} \right)^{2}  
	                  \, \right]^{2} + \left[ 2 \, ^{2} - \ \left( \frac{1}{w} \right)^{2}  \right]^{2} \           \\    
	         = \ & \ \frac{4 \cdot w^{2} - 1}{ w^{4}} \, +  \left[ \; 16 + \ \frac{1}{w^{4} } 
	            - \,  \frac{8}{w^{2}}  \, \right] \ = \, \frac{4}{w^{2}} + 16 \, - \, \frac{8}{w^{2}} \  ,            
	       \end{align} 
	       and it is trivial that the last sum is less than 16, hence the Cauchy Schwarz inequality 
	       \eqref{abschaetzungCSB} is confirmed for $ \mathsf{ (Case \, b) }$. \
         Before we deal with the last case $ \mathsf{ (Case \, c) }$, one more lemma is necessary. 
	 \begin{lemma}  \label{last lemma}  Let \quad $| t | , |w | \geq \frac{1}{2} $.   \\
	    \qquad The following two inequalities $ \mathsf{ (C)}$ and $\mathsf{ (D)}$ are equivalent, and both are true.  
	    \begin{align*}
	    {\mathsf{ (C)}}:   \quad & \left[  2 + \frac{ \sqrt{ 4 \cdot w^{2} - 1}}{w^{2}} - 
	                \left(\frac{1}{t}\right)^{2} \right]^{2} \ + \
	                \left[ 2 + \frac{ \sqrt{ 4 \cdot t^{2} - 1}}{t^{2}} - \left(\frac{1}{w}\right)^{2} \right]^{2}
	                \ \ \leq \ 16	                               \\
	    {\mathsf{ (D)}}:   \quad & \left( 2 \cdot t^{2} - 1 \right) \cdot \sqrt{ 4 \cdot w^{2} - 1} \ + \ 
	                \left( 2 \cdot w^{2} - 1 \right) \cdot \sqrt{ 4 \cdot t^{2} - 1}
	                \ \ \leq \ 4 \cdot t^{2} \cdot w^{2}             
	     \end{align*}  
	 \end{lemma}  
	 \begin{proof} Starting with $ \mathsf{ (C)}$, the proof of the equivalence is straightforward. 
	                                      
	       The last step is to confirm the second inequality  $\mathsf{ (D)}$ for all $| t | , |w | \geq \frac{1}{2} $.    
	       This needs two tricky substitutions. The first is \ 
	       $ p :=  4 \cdot t^{2} - 1 $ \ and \ $ z :=  4 \cdot w^{2} - 1 $. The inequality $\mathsf{ (D)}$ leads to 
	 \begin{align}
	            \left(\frac{p+1}{2} - 1 \right) \cdot \sqrt{z} \ + \ 
	            \left(\frac{z+1}{2} - 1 \right) \cdot \sqrt{p} \ \
	            \leq \ & \  \frac{1}{4} \cdot \left( p \cdot z + p + z + 1 \right) \  \\
	            \Longleftrightarrow   \ \ \ p \cdot \sqrt{z}  -  \sqrt{z} \ + \ z \cdot \sqrt{p}  - \sqrt{p}  \ \
	            \leq \ & \ \frac{1}{2} \cdot \left( p \cdot z + p + z + 1 \right) \ .  
	 \end{align}  
	     The  second substitution is \ $ h := \sqrt{p} $ \ and \ $ k := \sqrt{z} $. \ It follows the 
	     equivalent inequalities
	    \begin{align}
	            \ & \ \ \ h^{2} \cdot k - k +  k^{2} \cdot h - h \ \ \leq 
	            \  \ \frac{1}{2} \cdot \left( h^{2} \cdot k^{2} + h^{2} + k^{2} + 1 \right) \ & \  \\
	            \Longleftrightarrow   \ & \ \ 
	            \  h^{2} \cdot \left( k - \frac{1}{2} \cdot k^{2} -  \frac{1}{2} \right) \ + \
	             h \cdot ( k^{2} - 1 ) \ - \ \left( k + \frac{1}{2} \cdot k^{2} + \frac{1}{2} \right) & \leq \ 0  
	    \end{align}  
	  We multiply it by ($-2$), and we get       
	    \begin{align}    
	            \ & \ \ \  h^{2} \cdot \left( - 2 \cdot k + k^{2} +  1 \right) \ - \
	              2 \cdot h \cdot ( k^{2} - 1 ) \ + \ \left( 2 \cdot k + k^{2} + 1 \right) &  \geq \ \ & \ 0  \\  
	             \Longleftrightarrow   \ & \ \ 
	             \left[ \,  h \cdot ( k - 1) \ - \ (k + 1 ) \, \right] ^{2} \  \ & \ \ \geq \ \ & \ 0   
	    \end{align}    
	    Obviously, the last inequality is true. Hence, both inequalities  ${\mathsf{ (C)}}$ and  ${\mathsf{ (D)}}$
	     in the lemma  are also correct, for real \ $ t, w $ \, with  $| t | , |w | \geq \frac{1}{2} $.          
	 \end{proof}      
	 \begin{remark}    
	     In the above Lemma \ref{last lemma} in the second inequality $\mathsf{ (D)}$ it occurs 
	     equality if and only if \, $  h \cdot ( k - 1) = k + 1 $, or equivalently if 
	     the  two variables $ t $ and $ w $ fulfill the relation
	  $$  t \ = \ \frac{ w \cdot \sqrt{2} }{\sqrt{ 4 \cdot w^{2} - 1} - 1 } \ , \ \ \ 
	             \text{for} \ \ | w | \ \neq \ \frac{1}{2} \cdot \sqrt{2}
	             \ \ \text{and} \ | t | , |w | \geq \frac{1}{2} \, .                              $$       
	          Note that in this formula the variables \, $ t $ and $ w $ \, can be exchanged. 
	                                          
	   Further we remark that in inequality $\mathsf{ (D)}$, for 
	   \, $ | t | = \frac{1}{2} \cdot \sqrt{2} $ \, or \, $ |w | = \frac{1}{2} \cdot \sqrt{2} $ \, 
	    both sides of  $\mathsf{ (D)}$ have a constant difference of 1.  
	\end{remark}      
	   Now we regard the last case. \\
	  $ \mathsf{ (Case \, c) }$: Let \ $ \sqrt{2} / 2  \, < \, t, w $.  \\
	   From Proposition \ref{entscheidene proposition} we have the estimates  ${\mathsf{ (A)}}$ and ${\mathsf{ (B)}}$,
	   hence we get the inequality    
	            \begin{align*}
	                 \left| 4 \cdot \left< \left( \begin{array}{c} 1 \\ 0 \end{array} \right) \ | \ 
	                \left( \begin{array}{c} 0 \\ 1 \end{array}  \right)   \right> \right|^{2} 
	                 \ \ \leq \ \ \left[ \;  2 + \frac{\sqrt{4 \cdot w^{2} - 1}}{ w^{2}} 
	                                       - \left( \frac{1}{t} \right)^{2} \;  \right]^{2}  \ + \ 
	                 \left[ \;  2 + \frac{\sqrt{4 \cdot t^{2} - 1}}{ t^{2}} 
	                 - \left( \frac{1}{w} \right)^{2} \;  \right]^{2} .                        
	         \end{align*}
	    Together with Lemma  \ref{last lemma} this yields the last step to prove the Cauchy-Schwarz-Bunjakowsky
	     inequality,  since from inequality $ \mathsf{ (C)}$  in Lemma  \ref{last lemma}  follows the 
	     ${\mathsf{ CSB}}$ inequality,   i.e. Theorem \ref{main theorem} finally is confirmed. 
	  \end{proof} 
       We add a few propositions which we do not use 
       (except the lines  \eqref{entscheidene Gleichung},  \eqref{Abchaetzung 1durch s} and 
       \eqref{noch eine Gleichung}) in the proof of Theorem \ref{main theorem}.
       We define the map $G, \ G:  {\mathbb R} \times {\mathbb R} \longrightarrow  {\mathbb R}^{+} $, \\
	  $$ (a,b) \ \longmapsto \ \sqrt{(1-a)^2 + b^2} \ + \ \sqrt{(1-b)^2 + a^{2}} \ + \ \sqrt{a^2 + b^2} \ .  $$
	   We look for the infimum $ \underline{{\cal M}} $ of $ G $. This infimum $ \underline{{\cal M}} $
	   must be a minimum, since it is easy to see that it occurs in the closed  square \\
	  \centerline{$ \{ (a,b) \in  {\mathbb R} \times {\mathbb R} \ | \ 0 \leq a , b \leq 1 \} $. }
	  The values of $ G $ can be interpreted as the sum of three hypotenuses of three rectangle triangles.   
	  We are not able to find $ \underline{{\cal M}} $, but if we
	  restrict our search on the diagonal $ a = b $, we find there with elementary analysis at \   
    $ a = b = (3 - \sqrt{3}) \, / \, 6  \ \approx \ 0.211 $ \ the minimum ${{\cal M}}$, where      
	  $$  \underline{{\cal M}} \ \leq \ {\cal M} \ = \ \sqrt{2 + \sqrt{3}}   \ \ = \ \    
                                \left( \sqrt{2}  +  \sqrt{6}  \right) \, / \, 2  \ \approx \ 1.932 \ . $$ 
	  This also might be the true global minimum  $ \underline{{\cal M}} $ of the map $ G $. Please note that we just
	  considered here the special case $ w = 1 $ from line \eqref{definition R von b} in 
	  Proposition \ref{entscheidene proposition}.     
	 \begin{proposition}  \label{proposition drei Punkt drei}   
	          It holds
	    \begin{equation}   \label{inequality 3.5}
	           \frac{1}{s} \  \leq \  \   
             \begin{cases}  \underline{{\cal M}} \ & \mbox{for} \ \ \  \frac{1}{w} \leq 1       \\    
                   \underline{{\cal M}} \cdot \frac{1}{w} \ & \mbox{for} \ \ \ \frac{1}{w} > 1 \ , \ \ \\
              \end{cases}  \ \quad   
             \frac{1}{v} \  \leq \  \     
             \begin{cases}  \underline{{\cal M}} \ & \mbox{for} \ \ \ \frac{1}{t} \leq 1           \\    
                   \underline{{\cal M}} \cdot \frac{1}{t} \ & \mbox{for} \ \ \ \frac{1}{t} > 1 \ . \\
             \end{cases}  \ \quad        
      \end{equation}   
 	 \end{proposition} 
	   \begin{proof}   We show the first of the two inequalities. Please see the equation 
	       \begin{align}   \label{entscheidene Gleichung} 
	             \left( \begin{array}{c} 1 \\ 1 \end{array} \right) 
	             \ = \ \left[(1-a) - {\mathbf{ i}} \cdot b \right] \cdot
	             \left(\begin{array}{c} 1 \\ 0 \end{array}  \right) \ + \ 
	                  \ \left[(1-b) + {\mathbf{ i}} \cdot a \right] \cdot
	             \left(\begin{array}{c} 0 \\ 1 \end{array}  \right) \ + \ 
	             \ \left[a + {\mathbf{ i}} \cdot b \right] \cdot
	             \left(\begin{array}{c} 1 \\ -{\mathbf{ i}} \end{array}  \right) 
	         \end{align}  
	       which is true for arbitrary $ a , b \in   {\mathbb R} $. By the triangle inequality, we get 
	        \begin{equation}     \label{Abchaetzung 1durch s}  
	          1/s \ = \ \| ( 1, 1) \| \ \leq \ \sqrt{(1-a)^2 + b^2} \ + \  \sqrt{(1-b)^2 + a^2} \ + \ 
	          \sqrt{a^2 + b^2} \cdot  \| ( 1, -{\mathbf{ i}} ) \|  .      
	        \end{equation}   
	       Either $ \| ( 1, -{\mathbf{ i}} ) \| \leq 1$ or $ \| ( 1, -{\mathbf{ i}} ) \| > 1$.
	       It follows the first inequality in this proposition.    
	       The other inequality uses the corresponding equation of the next line \eqref{noch eine Gleichung}.    
	\begin{align}  \label{noch eine Gleichung}        
      (1, {\mathbf{ i}} ) \ = \ \left[(1-a) + {\mathbf{ i}} \cdot b \right] \cdot (1, 0 ) \ + \ 
      \left[ {\mathbf{ i}} \cdot (1-b) + a \right] \cdot (0, 1 ) \ + \ \left[a - {\mathbf{ i}} \cdot b \right] \cdot              (1,-1)  
  \end{align}  
 \end{proof}
   \begin{proposition}  \label{schon wieder eine Proposition}
	               Let \ $ s > 1/2 $ \ and \ $ v > 1/2 $, \ respectively. It holds   
	              $$   s  \, \leq \, \frac{s}{2 \cdot s - 1 } \ , \ \ \text{and} \ \
	                   v  \, \leq \, \frac{v}{2 \cdot v - 1 } \ , \  \ \text{respectively} \ .      $$  
	  \end{proposition}  
	  \begin{proof} We prove $ s  \leq  \frac{s}{2 \cdot s - 1 } $. This is equivalent to $ s \, \leq \, 1 $. \ 
	         Due to the proof of the following proposition this is always true.
	  \end{proof}   
 \begin{proposition}  \label{proposition eins}  
	       It holds both 
	          $$    s \ \leq \ 1  \ \ \ \text{and} \ \ \  v \ \leq \ 1 \ .                         $$  
 \end{proposition}  
  \begin{proof}  
	  Please see the equation \ $(s,0) = 1/2 \cdot (s,s) + 1/2 \cdot (s,-s)$. 
	       Note the norms \  $ \| (s,0) \| = s, \ \| (s,s) \| = 1 = \| (t,-t) \| $. 
	                Since $ s \leq t$ we have $ \| (s,-s) \| \leq 1 $.  Now apply the triangle inequality 
	                to the equation.     
	\end{proof}
	\begin{proposition}  Let \ $ s, t > 1/2 $. \ \    The following two tripels   \\
	     $ (t,-t) , (1,0), \left(\frac{t}{2 \cdot t - 1}, \frac{t}{2 \cdot t - 1} \right) $              \ and \
	     $ (s,s), (1,0), \left(\frac{s}{2 \cdot s - 1}, \frac{-s}{2 \cdot s - 1} \right) $    
	     are collinear. (On two different lines, of course, except the special cases \ $ 2 \cdot s \cdot t = s +t $). 
 	\end{proposition}
	\begin{proof} The first statement is proven by \ 
	$ (t,-t) - (1,0) \ = \ ( 1 - 2 \cdot t ) \cdot \left[ \frac{t}{2 \cdot t - 1 } \cdot (1,1) - (1,0) \right] $.  
	\ The second statement uses \ 
	$ (s,s) - (1,0) \ = \  ( 1 - 2 \cdot s ) \cdot \left[ \frac{s}{2 \cdot s - 1 } \cdot (1,-1) - (1,0) \right] $.  
	\end{proof}  
	  \begin{proposition} $ \min \{t,w\} \ \leq \ \sqrt{2} $.
	  \end{proposition}  
	  \begin{proof} 
	     We use $ e^{{\mathbf{ i}} \cdot \pi/4} $ 
	          from the complex unit circle, where
	   $$     e^{{\mathbf{ i}} \cdot 45 ^ \circ} \ = \ \ e^{{\mathbf{ i}} \cdot \pi/4}
	          \ = \ \cos (\pi/4) + {\mathbf{ i}} \cdot \sin (\pi/4) 
	          \ = \    \frac{1}{2} \cdot \sqrt{2} \ + \ {\mathbf{ i}} \cdot \frac{1}{2} \cdot \sqrt{2} \ .     $$  
	       First assume $ w \leq t $.  We write
	   \begin{displaymath}   \label{nocheinegleichung}
	            \frac{1}{2} \cdot \sqrt{2} \cdot w \cdot \left( \begin{array}{c} 1 \\ 0 \end{array} \right) 
	        \ = \ \  \left(  \frac{1}{2} \cdot \sqrt{2} \ + \ {\mathbf{ i}} \cdot \frac{1}{2} \cdot \sqrt{2}  \right)
	              \ \cdot \  \left[ \frac{1}{2} \cdot
	               \left(  \begin{array}{c} -{\mathbf{ i}} \cdot w \\ {\mathbf{ i}} \cdot w  \end{array}  \right) 
	                         \ + \  \frac{1}{2} \ \cdot \
	                         \left( \begin{array}{c} w \\ -{\mathbf{ i}} \cdot w  \end{array} \right) \right] \ .
	   \end{displaymath}  
	     By $ w \leq t $, we have the norms  
	     $ \| (-{\mathbf{ i}} \cdot w, {\mathbf{ i}} \cdot w) \| =  \| ( w, - w) \| \leq \| ( t, - t) \| = 1 $,
	     and \ $ \| ( w, -{\mathbf{ i}} \cdot w) \| = 1 $. \ By the triangle inequality it follows \ 
	     $   \frac{1}{2} \cdot \sqrt{2} \cdot w  \leq \frac{1}{2} + \frac{1}{2} $ . \ The case $ t \leq w $
	     is treated with the corresponding equation \\  
	     \centerline  {$ \frac{1}{2} \cdot \sqrt{2} \cdot t \cdot \left( 1, 0 \right) \ \ = \ \  
	     \left( \ \frac{1}{2} \cdot \sqrt{2} \ + \ {\mathbf{ i}} \cdot \frac{1}{2} \cdot \sqrt{2}   \right)
	      \ \cdot \ \left[ \frac{1}{2} \cdot \left( -{\mathbf{ i}} \cdot t , {\mathbf{ i}} \cdot t \right) 
	               \ + \ \frac{1}{2} \cdot \left( t , -{\mathbf{ i}} \cdot t \right) \right] $ \ .            } 
	  \end{proof}
   \begin{proposition}   
	              There are two inequalities \
	        $$     \frac{1}{s} \ \leq \ \frac{1}{v} + \frac{1}{t} + \frac{1}{w} \ , \quad \text{and} \quad \ 
	          \frac{1}{v} \ \leq \ \frac{1}{s} + \frac{1}{t} + \frac{1}{w} \ .        $$  
	 \end{proposition} 
	 \begin{proof}   For instance, \, $ 1/v \leq 1/s + 1/t + 1/w $ \, is proven by the following equation, \\
	     \centerline{ $ (1, {\mathbf{ i}}) = (1, 1) +  (1, -1 ) - (1, -{\mathbf{ i}}) $. } 
	 \end{proof}
   \begin{proposition}   
	    $$  \text{It holds} \qquad \ \frac{1}{s} - \frac{1}{t} \ \leq \ 2 \ \leq \ \frac{1}{s} + \frac{1}{t} \ \qquad
	        \text{and} \qquad  \ \frac{1}{v} - \frac{1}{w} \ \leq \ 2 \ \leq \ \frac{1}{v} + \frac{1}{w} \  .  $$  
	  \end{proposition} 
	  \begin{proof}   For the first inequality use two times the equation  \ \ $ (1, 1) = (2, 0) -  (1, -1) $. 
	  \end{proof}
	  \begin{proposition} Let \ $ \alpha \in \{s, t \} $, \ and \ $ \gamma \in \{v, w \} $.  
	    $$  \text{It holds} \qquad \ \left| \frac{1}{\alpha} - \frac{1}{\gamma} \right| \ \leq \ \sqrt{2} \ \leq \ 
	                                        \frac{1}{\alpha} + \frac{1}{\gamma}  \  .  $$  
	  \end{proposition} 
	   \begin{proof}   We need to show \ $ \frac{1}{v} - \frac{1}{t} \leq \sqrt{2} $ \ and \
	          $ \frac{1}{s} - \frac{1}{w} \leq \sqrt{2} $ \ and \ $ \sqrt{2} \leq \ \frac{1}{t} + \frac{1}{w} $.  \\ 
	    Note $ |1 + {\mathbf{ i}}| = |1 - {\mathbf{ i}}| = \sqrt{2} $, \ and consider 
	      \begin{displaymath}   \label{proegfhehastg} 
	         \   \left( \begin{array}{c} 1 \\ {\mathbf{ i}}  \end{array} \right)    
	              \ = \  \left(  \begin{array}{c} 1 \\ -1 \end{array}  \right) \ + \ 
	                   ( 1 + {\mathbf{ i}}) \cdot  \left(  \begin{array}{c} 0 \\ 1  \end{array}  \right) \ .       
	       \end{displaymath}  
	       Hence it follows \ $ \frac{1}{v} \ \leq \ \frac{1}{t} + \sqrt{2} $. \
	    Please see also \ \ $ (1,1) = ( 1 ,-{\mathbf{ i}}) + (1 + {\mathbf{ i}}) \cdot (0,1) $ \ and \  
	                      \ $  (1 - {\mathbf{ i}}) \cdot (0,1)  = ( 1 ,-{\mathbf{ i}}) - (1,-1) $ .
	   \end{proof}
	   \begin{proposition}  \label{noch eine Proposition}
	    $$  \text{It holds} \quad \ s \geq \sqrt{2} \ \cdot \left[ \frac{v \cdot w}{v + w} \right] \ \ \text{and} \ \
	                              \ v \geq \sqrt{2} \ \cdot \left[ \frac{s \cdot t}{s + t} \right] \ . $$  
	  \end{proposition} 
	   \begin{proof} 
	      We show \ $ s \geq \sqrt{2} \, \cdot \left[ \frac{v \cdot w}{v + w} \right] $. \ We use
	      \begin{displaymath}   \label{proo8oooo7astg} 
	              \left( \begin{array}{c} 1 \\ 1 \end{array} \right) 
	              \ = \ \frac{1 - {\mathbf{ i}}}{2 \cdot v} \ \cdot
	                \left(\begin{array}{c} v \\ {\mathbf{ i}} \cdot v \end{array} \right) \ + \ 
	                \ \frac{1 + {\mathbf{ i}}}{2 \cdot w} \cdot
	                \left(  \begin{array}{c} w \\ -{\mathbf{ i}} \cdot w  \end{array}  \right) \ .
	      \end{displaymath}    
	           The other inequality needs \ $ \left( \begin{array}{c} 1 \\ {\mathbf{ i}} \end{array} \right) 
	              \ = \ \frac{1 + {\mathbf{ i}}}{2 \cdot s} \ \cdot \
	                \left(\begin{array}{c} s \\ s \end{array} \right) \ + \ 
	                \ \frac{1 - {\mathbf{ i}}}{2 \cdot t} \cdot
	                \left(  \begin{array}{c} t \\ -t  \end{array}  \right) $ \ .     
	      	   \end{proof}
		  { \textbf{ Acknowledgements: }}     We wish to thank Prof. Dr. Eberhard  Oeljeklaus 
		  for a careful reading of the paper and some helpful calculations, and Dr. Malte von Arnim, 
 		  who found the substitutions in the proof of inequality  $\mathsf{ (D)}$ in Lemma \ref{last lemma}.  

  \end{document}